\documentclass{article}
\usepackage[utf8]{inputenc}
\usepackage{graphicx}
\usepackage{hyperref}
\usepackage{amsthm}
\usepackage{amssymb}
\usepackage{booktabs}
\usepackage{mathtools}
\usepackage{a4wide}
\usepackage{siunitx}
\usepackage{todonotes}
\usepackage{enumitem}
\usepackage{subcaption}

\usepackage{babel}
\usepackage{listings}

\usepackage{gnuplottex}

\theoremstyle{plain}
\newtheorem{thm}{\protect\theoremname}
\theoremstyle{definition}
\newtheorem{defn}[thm]{\protect\definitionname}
\theoremstyle{remark}
\newtheorem{rem}[thm]{\protect\remarkname}
\theoremstyle{plain}
\newtheorem{lem}[thm]{\protect\lemmaname}
\theoremstyle{definition}
\newtheorem{example}[thm]{\protect\examplename}

\newcommand{\algo}[1]{\textsf{#1}}
\newcommand{\predicate}[1]{\textsf{#1}}
\newcommand{\R}{\mathbb{R}}

\providecommand{\definitionname}{Definition}
\providecommand{\examplename}{Example}
\providecommand{\lemmaname}{Lemma}
\providecommand{\remarkname}{Remark}
\providecommand{\theoremname}{Theorem}

\title{Fast Floating-Point Filters for Robust Predicates}
\author{Tinko Bartels\thanks{Technical University of Berlin, Straße des 17. Juni 135, 10623 Berlin, Germany, t.bartels@tu-berlin.de} 
        \and Vissarion Fisikopoulos\thanks{Oracle, Greece,  vissarion.fysikopoulos@oracle.com}
        \and Martin Weiser\thanks{Zuse Institute Berlin, Takustr. 9, 14195 Berlin, Germany, weiser@zib.de}}
\date{\today}

\begin{document}

\maketitle

\abstract{
Geometric predicates are at the core of many algorithms, such as the construction of Delaunay triangulations, mesh processing and spatial relation tests.
These algorithms have applications in scientific computing, geographic information systems and computer-aided design.
With floating-point arithmetic, these geometric predicates can incur round-off errors that may lead to incorrect results and inconsistencies, causing computations to fail. 
This issue has been addressed using a combination of exact arithmetic for robustness and floating-point filters to mitigate the computational cost of exact computations.
The implementation of exact computations and floating-point filters can be a difficult task, and code generation tools have been proposed to address this. 
We present a new C++ meta-programming framework for the generation of fast, robust predicates for arbitrary geometric predicates based on polynomial expressions. 
We combine and extend different approaches to filtering, branch reduction, and overflow avoidance that have previously been proposed.
We show examples of how this approach produces correct results for data sets that could lead to incorrect predicate results with naive implementations. 
Our benchmark results demonstrate that our implementation surpasses state-of-the-art implementations.
}

\section{Introduction}

Basic geometric predicates, such as computing the orientation of a triangle or testing if a point is inside a circle, are at the core of many computational geometry algorithms such as convex hull, Delaunay triangulation and mesh generation~\cite{Berg08}. 
Interestingly, those predicates also appear in geospatial computations such as topological spatial relations that determine the relationship among geometries. 
Those operations are fundamental in many Geographic Information System (GIS) applications. 
If evaluated with floating-point arithmetic, these computations can incur round-off errors that can, due to the ill-conditioning of discrete decisions, lead to incorrect results and inconsistencies, causing computations to fail~\cite{Kettner2004}. 

Among other applications, Delaunay triangulations are important for the construction of triangular meshes~\cite{Jamin15,Shewchuk98} and Triangulated Irregular Networks (TIN)~\cite{Li2004}. 
Predicate failures in the underlying Delaunay triangulation may lead to suboptimal mesh quality and cause invalid triangulations or termination failure~\cite{Shewchuk1997}. 

Robust geometric predicates can also be used in spatial predicates to guarantee correct results for floating-point geometries. 
Spatial predicates are used to determine the relationship between geometries and have applications in spatial databases and GIS applications. 
Examples of such predicates include \predicate{intersects}, \predicate{crosses}, \predicate{touches}, or \predicate{within}. 
Using non-robust spatial predicates, for example, a point that lies close to the shared edge of two triangles can be found to be within both or neither of them, which is not only incorrect but also inconsistent and violates basic assumptions on partitioned spaces. 

Exact computations can guarantee correct results for floating-point input but are very slow for practical purposes. Since predicates are usually ill-conditioned only on a set of measure zero and extremely well-conditioned everywhere else, an adaptive evaluation can improve average performance by using exact arithmetic only if an a priori error estimate can not guarantee correctness for the faster, approximate computations. 
In other words, the expensive computations are filtered out by using those error estimates. 

Now, the main question is how difficult it is to compute those error estimates. 
There are several approaches that provide a trade-off in efficiency and accuracy of error estimation.
The three main types of filters are static, semi-static and dynamic. 

In the first case, the error is pre-computed very efficiently using a priori bounds on the input and typically attains very low accuracy. 
In semi-static filters, the error estimation depends on the input. 
They are somewhat slower than static filters but improve on the accuracy and require no a priori bounds on the input. 
The slowest and most accurate are dynamic filters using floating-point interval arithmetic to better control the error and achieve fewer filter failures.

\paragraph{Previous work.} 
Many techniques have been proposed in the past for efficient and robust arithmetic. 
In his seminal paper~\cite{Shewchuk1997}, Shewchuk introduced robust, adaptive implementations for \predicate{orientation}-, \predicate{incircle}- and \predicate{insphere}-predicates that can be used, for example, in the construction of Delaunay triangulations.
They use a sequence of semi-static filters of ever-increasing accuracy.
The phases are attempted in order, each phase building on the result from the previous one until the correct sign is obtained.
On the other hand, efficient dynamic filters are proposed in~\cite{Bronnimann}.
For Delaunay triangulations, in~\cite{Devillers03} they propose a set of efficient static and semi-static filters and experimentally compare them with several alternatives including~\cite{Shewchuk1997}.
Meyer and Sylvain develop FPG~\cite{meyer:inria-00344297}, a general-purpose code analyzer and generator for 
static filtered predicates. 
The generated filters, however, include multiple branch instructions, which was found in~\cite{Ozaki2016} to cause suboptimal performance.

Nanevski et al. extend Shewchuk's method to arbitrary polynomial expressions and implement an expression compiler that takes a function and produces a predicate, consisting of semi-static filters and an exact stage that computes the sign of the source function at any given floating point arguments~\cite{Nanevski2003}. Their filters, however, are not robust with respect to overflow or underflow.

In~\cite{burnikel_exact_2001}, Burnikel et al. present EXPCOMP, a C++ wrapper class and an expression compiler that generates fast semi-static filters for predicates involving the operations $+, -, \cdot, /, \sqrt{\cdot}$, which handle all kinds of floating-point exceptions. In their benchmarks, they found a 25-30\% runtime overhead for their C++ wrapper class when compared with their expression compiler.
More recently, Ozaki et al. developed an improved static filter as well as a new semi-static filter for the 2D orientation predicate, where the latter also handles floating-point exceptions such as overflow and underflow~\cite{Ozaki2016}. 

Their filters, however, are not designed for arbitrary polynomial predicates.
Regarding non-linear geometries, there is work on filters for circular arcs~\cite{Devillers00}.
Moreover, robust predicates could be extended to provide robust constructions such as points of intersection of linestrings~\cite{Attene}.
Recently, GPU implementations of robust predicates have been presented, providing a constant (3 to 4 times) speedup over standard CPU implementations~\cite{Meng19,Menezes2021FastPE}.

In~\cite{fisikopoulos16}, they employ dynamic determinant computations to speed up the computation of sequences of determinants that appear in high-dimensional (typically more than 6) geometric algorithms such as convex hull and volume computation. 

In~\cite{bartels21}, the authors present a C++ metaprogramming framework that produces fast, robust predicates and illustrate how GIS applications can benefit from it. 

\paragraph{Our contribution.}
The contribution of this paper is three-fold. 
First, we present an algorithm that generates semi-static or static filters for robust predicates based on arbitrary polynomials. These filters are shown to be valid for all input numbers, regardless of range issues such as overflow or underflow. They also require only a single comparison and can therefore be evaluated encountering only a single, easy-to-predict branch.
To the best of our knowledge, this is the first filter design combining generality, range robustness, and branch efficiency.

Second, we present a new implementation based on C++ meta-programming techniques that produces fast, robust code at compile-time for predicates.
It is extensible, based on the \href{https://github.com/boostorg/geometry}{C++ library Boost.Geometry}~\cite{BG} and publicly available 
\begin{center}
\url{https://doi.org/10.5281/zenodo.6836062}. 
\end{center}
The main advantage of our implementation is the ability to automatically generate filters for arbitrary polynomial predicates without relying on external code generation tools. 
In addition, it can be complemented seamlessly with manual handcrafted filters, as illustrated by the use of our axis-aligned filter for the \predicate{incircle} predicate (see example~\ref{exa:incircle-rect}).

Last, we perform an experimental analysis of the generated filters as well as a comparison with the state of the art. 
We perform benchmarks for 2D Delaunay triangulation, 3D Polygon Mesh processing and 3D Mesh refinement.
The algorithms tested in the benchmarks make use of four different geometric predicates of different complexity.
We show that our predicates outperform the state of the art  libraries~\cite{shewchukimpl,cgal:bfghhkps-lgk23-21a} in all benchmark cases, which includes both synthetic and real data. Unlike Burnikel et al. in~\cite{burnikel_exact_2001} we find no performance penalty for our C++ implementation over generated code.

\section{Robust Geometric Predicates}\label{robust geometric predicates}

In this section we review the basic concepts and notation necessary for presenting our filter design in  Section~\ref{sec:semi-static-filters} and implementation approach in Section~\ref{sec:metaprogramming}. 

\subsection{Geometric Predicates and Robustness Issues}\label{geometic predicates and robustness issues}

In the context of this paper, we define geometric predicates to be functions that return discrete answers to geometric questions based on evaluating the sign of a polynomial. 
One example is the planar \predicate{orientation} predicate. 
Given three points $a,b,c \in \R^2$, it determines the location of $c$ with respect to the straight line going through $a$ and $b$ by evaluating the sign of 
\begin{equation}
    p_{\text{orientation\_2}}\left(a,b\right)\coloneqq\left|
    \begin{array}{cc}
        a_{x}-c_{x} & a_{y}-c_{y}\\
        b_{x}-c_{x} & b_{y}-c_{y}
    \end{array}\right|\label{eq:2d_orientation_expression_real}
\end{equation}

For this definition of the \predicate{orientation} predicate, positive, zero, and negative determinants correspond to the locations left of the line, on the line and right of the line, respectively. 
This geometric predicate has applications in the construction of Delaunay triangulations, convex hulls, and in spatial predicates such as \predicate{within} for 2D points, lines or polygons.

While expression (\ref{eq:2d_orientation_expression_real}) always gives the correct answer in real arithmetic, this is not necessarily the case for floating-point arithmetic.

\begin{defn}[Floating-Point Number System] \label{def:binary-fpn}
For a given precision $p \in \mathbb{N}_{\geq 2}$ and minimum and maximum exponents $e_{\min}, e_{\max}\in\mathbb{Z}$ we define by 
\[
N_{p,e_{\min},e_{\max}} \coloneqq \{ \left(-1\right)^\sigma \left( 1 + \sum_{i=1}^{p-1} b_i 2^{-i} \right) 2^e \mid \sigma , b_1,\ldots,b_p \in \{ 0, 1 \}, e_{\min} \leq e \leq e_{\max}  \}
\]
the set of \emph{normalised binary floating-point numbers} and by 
\[
S_{p,e_{\min}} \coloneqq \{ \left(-1\right)^\sigma \left( \sum_{i=1}^{p-1} b_i 2^{-i} \right) 2^{e_{\min}} \mid \sigma , b_1,\ldots,b_p \in \{ 0, 1 \} \}
\]
the set of \emph{subnormal binary floating-point numbers}. 

For the remainder we will drop the parameters in the subscript. 
A \emph{binary Floating-Point Number system} (FPN) is defined by $F \coloneqq N \cup S \cup \{ -\infty,\infty,\text{NaN} \}$. For a number $a \in F$ given in the representation 
\[
\left(-1\right)^\sigma \left( 1 + \sum_{i=1}^{p-1} b_i 2^{-i} \right) 2^e
\text{ or }
\left(-1\right)^\sigma \left( \sum_{i=1}^{p-1} b_i 2^{-i} \right) 2^{e_{\min}},
\]
we call the tuple $\left(b_1, \ldots, b_{p-1}\right)$ \emph{significand}. It is sometimes called \emph{mantissa} in literature. 
The significand is called \emph{even} if $b_{p-1} = 0$.

\begin{defn}[Rounding function]
For a given FPN $F$ we define the \emph{rounding-function} $\text{rd}:\mathbb{R}\rightarrow F$ as follows

\[
\text{rd}\left(a\right)\coloneqq\begin{cases}
-\infty & a\leq-\left(2^{e_{\max}}-2^{e_{\max}-p}\right)\\
\text{closest number to \ensuremath{a} in \ensuremath{F}} & -\left(2^{e_{\max}}-2^{e_{\max}-p}\right)<a<2^{e_{\max}}-2^{e_{\max}-p}\\
\infty & a\geq2^{e_{\max}}-2^{e_{\max}-p}\\
\end{cases}
\]

If there are two nearest numbers in $F$, the one with an even significand is chosen. 
\end{defn}\label{def:floating-point-ops}

 Next, we define some special quantities. By $\varepsilon\coloneqq2^{-p}$ we denote the \emph{machine epsilon}, which is half of the difference between $1.0$ and the next number in $F$,
 by $u_{N}\coloneqq2^{-e_{\min}}$ the smallest, positive, normalized
 number in $F$ and by $u_{S}\coloneqq2^{-e_{\min}-p+1}=2\cdot\varepsilon\cdot u_{N}$
 the smallest, positive, subnormal number in $F$.

\begin{defn}[Floating-point operations]
For a given FPN $F$ and $a, b \in F \cap \mathbb {R}$ we define the floating-point operator 
$\circledcirc : F\times F \rightarrow F$ for each $\circ \in \{+,-,\cdot \}$ 
such that 
$ \left| a \circledcirc b - a \circ b \right| \leq \varepsilon \left|a \circ b\right| $ for $\circ \in \{+,- \}$ and 
$ \left| a \odot b - a \cdot b \right| \leq \varepsilon \left|a \cdot b\right| + \frac{1}{2} u_{\rm{S}} $ 
if the result using the corresponding operator on $\mathbb{R}$ would be rounded to a number in $ F \cap \mathbb{R}$. 
Otherwise, the floating-point operators produce the corresponding signed infinity, which is called an \emph{overflow}. 
If a zero is produced in the floating-point multiplication of two non-zero numbers or a subnormal number is produced, this is called an \emph{underflow}. 
If no underflow occurs in floating-point multiplication, we require $ \left| a \odot b - a \cdot b \right| \leq \varepsilon \left|a \cdot b\right|$. 
These definitions are extended to arguments with NaN by setting the result to NaN and to infinities in the natural way with the following special cases set to NaN: $\infty \oplus -\infty$, $-\infty \ominus -\infty$, $\infty \ominus \infty$, $\pm \infty \odot 0$.
\end{defn}

This model of floating-point operations is given as the standard model with gradual underflow in chapter two of~\cite{Higham2002} and~\cite{Rump2011}.

\end{defn}

Common examples include the binary FPN with $p = 24, e_{\min} = -126, e_{\max} = 127$ called \emph{single-precision} or \emph{FP32} and the binary FPN with $p = 53, e_{\min} = -1022, e_{\max} = 1023$ \emph{double-precision} or \emph{FP64}.

\begin{rem}
The requirements of the previous definition are met by IEEE 754-conformant binary floating-point number systems which include the native single- and double-precision
floating-point types of the architectures x86, x86-64, current ARM, common
virtual machines running WebAssembly and current CUDA processors.
The machine epsilon is sometimes defined as the
difference between $1.0$ and the next number in $F$.
\end{rem}

We call 
\begin{align}
    \tilde{p}_{\text{orientation\_2}}\left(a,b,c\right)\coloneqq& \left(a_{x}\ominus c_{x}\right)\odot\left(b_{y}\ominus c_{y}\right) \ominus\nonumber\\&\left(a_{y}\ominus c_{y}\right)\odot\left(b_{x}\ominus c_{x}\right)\label{eq:2d_orientation_expression_fpn}
\end{align}
a floating-point realisation of (\ref{eq:2d_orientation_expression_real}). 
Due to rounding errors, this realisation can produce incorrect results.

As an example, consider the points $a=\left(-0.01, -0.59\right)$, $b=\left(0.01, 0.57\right)$, $c=\left(0,-0.01\right)$. 
In real arithmetic, $c$ lies on the straight line through $a$ and $b$. 
Their closest approximations in $F_{53,-1022,1023}$ (IEEE 754-2008 binary64~\cite{IEEE754}, or FP64 for short), $\tilde{a}, \tilde{b}, \tilde{c}$, however, are only very close to collinear, which makes the case sensitive to rounding errors. 

As a second example, let us evaluate the spatial relationship between the point $c$ and the closed triangles $\tilde{t}_1 \coloneqq \{\left( -1, 0 \right) , \tilde{a}, \tilde{b} \}$ and $\tilde{t}_2 \coloneqq \{ \left( 1, 0 \right) , \tilde{b}, \tilde{a} \}$ using the winding-number algorithm~\cite{Sunday21}.
\begin{table}\centering
\begin{tabular}{ lccc  }
 \toprule
    Architecture & $\tilde{c}$ and $\tilde{t}_1$ & $c$ and $\tilde{t}_2$ & $c$ and $\tilde{t}_1 \cup \tilde{t}_2$ \\
 \midrule
    -march=haswell & outside & outside & inside \\
    -march=ivybridge & touches & touches & inside\\
    exact & inside & outside & inside\\
 \bottomrule
\end{tabular}
\caption{Relationships of point $c=\left(0,-0.01\right)$ to polygon $\tilde{t}_1 \coloneqq \{\left( -1, 0 \right), \tilde{a}, \tilde{b} \}$ and $\tilde{t}_2 \coloneqq \{ \left( 1, 0 \right) , \tilde{b}, \tilde{a} \}$, where $a=\left(-0.01, -0.59\right)$, $b=\left(0.01, 0.57\right)$.\label{tbl:point-polygon-rel}}
\end{table}

Table~\ref{tbl:point-polygon-rel} summarizes the results, all compiled with GCC 11.1 and optimization level O2. 
The first row is particularly noteworthy because the results are not only incorrect but also mutually contradictory. 
The final row can be obtained using any implementation of the \predicate{orientation} predicate that guarantees correct results, such as the implementation of Shewchuk~\cite{shewchukimpl} or CGAL's kernels \algo{epick} or \algo{epeck}~\cite{cgal:bfghhkps-lgk23-21a}. 
\begin{rem}
    The difference between the architectures is due to GCC producing an assembly code using the Fused Multiply-Add (FMA) instruction for evaluating~\eqref{eq:2d_orientation_expression_real}.
    FMA can be defined as 
    \[
    \text{FMA}\left(a, b, c\right) \coloneqq \text{rd}\left(a \cdot b + c\right).
    \]
    This instruction causes loss of anticommutativity for difference, i.e. $a \odot b\ominus c \odot d = - (c \odot d\ominus a \odot b)$ holds if no range errors occur, but $\text{FMA}\left(a,b,-c \odot d\right) = -\text{FMA}\left(c,d,-a \odot b\right)$ is not necessarily true. 
    When inserted into the \predicate{orientation} predicate, this can lead to situations in which swapping two input points does not reverse the sign of the result. 
\end{rem}
Inconsistencies can occur without FMA as well. 
Consider $\tilde{a}, \tilde{b}, \tilde{d} \coloneqq \left( \text{rd}\left(0.15\right), \text{rd}\left(8.69 \right) \right)$ and $\tilde{e} \coloneqq \left( \text{rd}\left(0.07\right), \text{rd}\left(4.05 \right) \right)$.
The floating-point realisation (\ref{eq:2d_orientation_expression_fpn}), compiled without FMA-optimizations, will determine $\tilde{a}, \tilde{b}, \tilde{e}$ and $\tilde{b}, \tilde{d}, \tilde{e}$ to be collinear but not $\tilde{a}, \tilde{b}, \tilde{d}$, which is a contradiction. 

Besides rounding errors, incorrect predicate results can also be caused by overflow or underflow. It can be easily checked that, in the FP64 number system,
\[
\tilde{p}_{\text{orientation\_2}}\left(\left(2^{-801}, 2^{-801}\right),\left(2^{-800}, 2^{-800}\right),\left(2^{-801}, 2^{-800}\right)\right) = 0,
\]
due to underflow, and 
\[
\tilde{p}_{\text{orientation\_2}}\left(\left(2^{800}, 2^{800}\right),\left(2^{800}, 2^{800}\right),\left(0, 0\right)\right) = \text{NaN},
\]
due to overflow.

Different approaches have been developed to obtain consistent results. We briefly discuss arbitrary precision arithmetic and floating-point filters in the following sections.

\subsection{Exact Arithmetic}\label{sec:exact arithmetic}

A natural idea to solve the precision issues of floating-point arithmetic would be to perform the computations at higher precision.
There are a number of arbitrary-precision libraries that implement number types with increased precision in software, such as GMP~\cite{Granlund12}, the CGAL Number Types package~\cite{cgal:hhkps-nt-22a} or Boost Multiprecision~\cite{BM}.

In combination with filters, such arbitrary-precision number types are used for exact geometric predicates in the CGAL 2D and 3D kernels, which were documented in \cite{cgal:bfghhkps-lgk23-21a}.
A drawback of software-implemented number types is that basic operations can be orders of magnitude slower than hardware-implemented operations for native number types such as single- or double-precision floating-point operations on most modern processor architectures.

An approach for arbitrary-precision arithmetic that makes use of hardware acceleration is expansion arithmetic.
A floating-point expansion is a tuple of multiple floating-point numbers that can represent a single number as an unevaluated sum with greater precision than a single floating-point number.
Because the operations on floating-point expansions are implemented in terms of hardware-accelerated operations on the components, they can be faster than more general techniques for arbitrary precision arithmetic.
The use of floating-point expansions for exact geometric predicates has been described in~\cite{Shewchuk1997}.

\subsection{Floating-Point Filters}\label{sec:floating-point filters}

We call an implementation a robust floating-point predicate if it is guaranteed to produce correct results. 
With expansion arithmetic, we can produce a robust predicate from a floating-point realisation by replacing all rounding floating-point operators $\oplus,\ominus$ and $\odot$ with the respective exact algorithms on floating-point expansions. 
The sign of the resulting expansion is then equal to the sign of its most significant (i.e. largest non-zero) component. 

The issue with this naive approach is that even simple predicates become computationally expensive. 
To mitigate this issue, we resort to expansion arithmetic only in the rare case that the straightforward floating point implementation is not guaranteed to produce the correct result. This decision is made by filters. 

\begin{defn}[Filter]
For a predicate $\text{sign}\left(p\left(x_{1},\ldots,x_{n}\right)\right)$
and an FPN system $F$, we call $f:M\subseteq F^{n}\rightarrow\left\{ -1,0,1,\text{uncertain}\right\} $
a \emph{floating-point filter}. $f$ is called valid for $p$ on $M$ if
for each $\left(x_{1},\ldots,x_{n}\right)\in M$ either $f\left(x_{1},\ldots x_{n}\right)=\text{sign}\left(p\left(x_{1},\ldots,x_{n}\right)\right)$
or $f\left(x_{1},\ldots,x_{n}\right)=\text{uncertain}$ holds. The
latter case is referred to as \emph{filter failure}.
\end{defn}

Adopting the terminology used in~\cite{Devillers03}, we call filters \emph{dynamic} if they require the computation of an error at every step of the computation, \emph{static} if they use a global error bound that does not depend on the inputs for each call of the predicate, and \emph{semi-static} if their error bound has a static component and a component that depends on the input. 
A variation of static filters, which require a priori restrictions on the inputs to compute global error bounds, are \emph{almost static} filters, which start with an error bound based on initial bounds on the input and update their error bound whenever the inputs exceed the previous bounds.

\begin{example}[Shewchuk's Stage A \predicate{orientation} predicate] \label{exa:2d-orientation-shewchuk-stage-a-filter} 
Consider the predicate $\mathrm{sign}(p)$ based on \eqref{eq:2d_orientation_expression_real} and its floating-point realisation $\mathrm{sign}(\tilde p)$~\eqref{eq:2d_orientation_expression_fpn}.
Then, 
    \begin{align*}
        f\left(a_{x},\ldots,c_{y}\right)\coloneqq\begin{cases}
            \text{sign}\left(\tilde{p}\right), & \text{if } \left|\tilde{p}\right|\geq e\left(a_{x},\ldots,c_{y}\right) \\
            \text{uncertain}, & \text{otherwise}
    \end{cases}
    \end{align*}
    with the error bound 
    \begin{align*}
        e\left(a_{x},\ldots,c_{y}\right) \coloneqq & \left(3\epsilon+16\epsilon^{2}\right)\odot(\left|\left(a_{x}\ominus c_{x}\right)\odot\left(b_{y}\ominus c_{y}\right)\right|\oplus\\
        & \left|\left(a_{y}\ominus c_{y}\right)\odot\left(b_{x}\ominus c_{x}\right)\right|),
    \end{align*}
    where $\tilde{p}\coloneqq\tilde{p}\left(a_{x},\ldots,c_{y}\right)$ and $\epsilon$ is the machine-epsilon of the FPN, is a valid filter for all inputs that do not cause underflow~\cite{Shewchuk1997}.
    
    If underflow occurs, however, validity is not guaranteed. Consider
the example 
\begin{align*}
a\coloneqq & \left(\begin{array}{c}
0\\
0
\end{array}\right)\\
b\coloneqq & \left(\begin{array}{c}
2^{e_{\min}}\\
0
\end{array}\right)\\
c\coloneqq & \left(\begin{array}{c}
2^{e_{\min}}\\
2^{e_{\min}}
\end{array}\right).
\end{align*}
Clearly the points are not collinear, however, $e\left(a_{x},\ldots,c_{y}\right)$
and $\tilde{p}$ will evaluate as zero due to underflow, which shows
that the filter can certify incorrect signs.
\end{example}
This filter can be considered semi-static, with its static component being $3\epsilon+16\epsilon^{2}$. 
The error bound is obtained mostly by applying standard forward-error analysis to the floating-point realisation. 
Shewchuk also described similar filters for the 2D \predicate{incircle} predicate, as well as the 3D \predicate{orientation} and \predicate{incircle} predicates.

\begin{example}[FPG \predicate{orientation} filter~\cite{meyer:inria-00344297}]\label{exa:2d-orientation-fpg-filter} Consider predicate (\ref{eq:2d_orientation_expression_real}) and its floating-point realisation (\ref{eq:2d_orientation_expression_fpn}).
    Let $m_{x} \coloneqq \max\{\left|a_{x}\ominus c_{x}\right|, \left|b_{x} \ominus c_{x} \right| \}$ and $m_{y} \coloneqq \max\{\left|a_{y}\ominus c_{y}\right|, \left|b_{y} \ominus c_{y} \right| \}$. 
    If 
    \[
    \max{\{m_{x}, m_{y}\}} > 2^{509},
    \] 
    \[
    0 \neq \min{\{m_{x}, m_{y}\}} \leq 2^{-485}
    \] 
    or \[\left|\tilde{p}\right| \leq \num{8.88720573725927e-16} \odot m_{x} \odot m_{y} \neq 0,\] 
    then "uncertain" is returned, otherwise the sign of $\tilde{p}$ is returned.
    The filter is valid with FP64 arithmetic for all FP64 inputs. It is also semi-static with the static component of the error bound being $\num{8.88720573725927e-16}$ (roughly $4\cdot\varepsilon$).

\end{example}

A static version of this filter can be obtained if global bounds for $m_x$ and $m_y$ are known a priori. 
The first two conditions are range-checks that guard against overflow and underflow. 
Apart from these conditions, the filter is based on an error bound similar to the previous example. 
The program FPG can generate such filters for arbitrary homogeneous polynomials if group annotations for the input variables are provided. 
In this context, group annotations are
lists of grouped variables that are part of the input for FPG.
The group annotations help the code generator with the choice of the scaling factors $m_{x}$ and $m_{y}$.
In the example above, the group annotations specified that $a_x, b_x$ and $c_x$ as well as $a_y, b_y$ and $c_y$ form a group. 

Another example of a semi-static error bound filter for the 2D \predicate{orientation} predicate that can handle overflow, underflow, and rounding errors with fewer branches than the filter generated by FPG can be found in~\cite{Ozaki2016}.

The next example is not strictly an error bound filter. 

\begin{example}[Shewshuk's stage B \predicate{orientation} predicate]\label{exa:2d-orientation-stage-b} 
Consider the predicate~\eqref{eq:2d_orientation_expression_real} and its floating-point realisation~\eqref{eq:2d_orientation_expression_fpn}.
    Let $d_{ax} \coloneqq a_{x}\ominus c_{x}, d_{bx} \coloneqq b_{x} \ominus c_{x}$ and analogously for y. 
    If the computations of these values incurred round-off errors, return uncertain. 
    Otherwise compute $d_{ax} \cdot d_{by} - d_{ay} \cdot b_{dx} $ exactly, using expansion arithmetic, and return the sign.
    This filter is described as stage B in \cite{Shewchuk1997} and is valid for all inputs that do not produce overflow or underflow. 
    The full version in \cite{Shewchuk1997} also includes an error bound check that allows preventing a filter failure if the no-round-off test fails.
\end{example}

Similar filters were presented by Shewchuk for other predicates. 
This filter is particularly effective for input points that are closer to each other than to $\left(0, 0\right)$ because differences of floating-point numbers that are within half/double of each other do not incur round-off errors. 
In the context of Shewchuk's multi-staged predicates, this filter also has the advantage that it can reuse computations from stage A and that its interim results can be reused for more precise stages in case of filter failure.
As a final example, we present a dynamic filter.

\begin{example}[Interval arithmetic filter]
    \label{exa:interval-filter} Consider a predicate and one of its floating-point realisations.
    Given the inputs, compute for each floating-point operation $\oplus, \ominus, \odot$ the lower and the upper bound of the result, including the rounding error, using interval arithmetic.
    If the final resulting interval contains numbers of different signs, return uncertain. Otherwise, return the shared sign of all numbers in the result interval.
    This approach is presented in~\cite{Bronnimann}.
\end{example}

In \cite{Devillers03}, \cite{Ozaki2016} and \cite{Shewchuk1997} failure probabilities and performance experiments for various sequences of filters, types of inputs and algorithms are presented. 
We will present our own experiments in Section~\ref{subsec:benchmarks}.

\section{Semi-Static Filters} \label{sec:semi-static-filters}

In this section, we will define a set of rules that allow us to derive error bounds for arbitrary floating-point polynomials. These error bounds will then be used to define semi-static filters. We start with establishing some properties of floating-point operations that will be used in the proof of the validity of our error bounds.

\begin{lem} \label{lem:fpn-error-bounds}
Let $a,b\in F$ be floating-point numbers.
\begin{enumerate}
\item If either $a$ or $b$ is in $\left\{ -\infty,\infty,\mathrm{NaN}\right\} $,
then 
\[
a\circledcirc b\in\left\{ -\infty,\infty,\mathrm{NaN}\right\} 
\]
 and 
\[
\left|a\circledcirc b\right|\in\left\{ \infty,\mathrm{NaN}\right\} 
\]
for every $\circledcirc\in\left\{ \oplus,\ominus,\odot\right\} $.
Consequently, the same holds for all floating-point expressions using
the operators $\oplus,\ominus,\odot$ and $\left|\cdot\right|$ that
contain a subexpression that evaluates to $-\infty,\infty$ or $\mathrm{NaN}$.
\item If an underflow occurs in the computation of $a \oplus b$ or $a \ominus b$, then the result is exact.
\end{enumerate}
\end{lem}

The first statement follows directly from~\ref{def:floating-point-ops} and the second statement is given as Theorem 3.4.1 in~\cite{hauser_handling_1996}.

Let $p:\mathbb{R}^{m}\rightarrow\mathbb{R}$ a polynomial in $m$
variables, denoted as $p\in \mathbb{R}\left[x_1,\ldots,x_m\right]$.
Let $\tilde{p}:F^{m}\rightarrow F$ be a floating-point realisation of $p$, i.e. a function on $F^{m}$ involving only the floating-point operations $\oplus,\ominus$ and $\odot$ such that it would be equivalent to $p$ if the floating-point operations were replaced by the corresponding exact operations.
Note, that $\tilde{p}$ is not unique, e.g. $\left( x_1 \oplus x_2 \right) \oplus x_3$ is different from $x_1 \oplus \left(x_2 \oplus x_3 \right)$ but both are floating-point realisations of the real polynomial $x_1 + x_2 + x_3$.
We denote by $F\left[x_1,\ldots,x_m\right]$ the set of floating-point realisations of polynomials in $m$ variables.
The subexpressions of $\tilde{p}$ will be denoted by $\tilde{p}_{1},\ldots,\tilde{p}_{k}$.

We will present a recursive scheme that allows the derivation of error
bound expressions for semi-static, almost static and static floating-point
filters. We will assume that the final operation of $\tilde{p}$ is
a sum or difference, so it holds that $\tilde{p}=\tilde{p}_{1}\circledcirc\tilde{p}_{2}$
with $\circledcirc\in\left\{ \oplus,\ominus\right\} $. If it were
a multiplication, the signs of each factor could be determined independently.
These filters will require only one branch like the filters in \cite{Ozaki2016}
and will not certify incorrect values for inputs that cause overflow
and optionally for inputs that cause underflow.

\subsection{Error Bounds}

As a reminder, semi-static error bounds are partially computed at compile-time and partially computed from the input values at runtime.
The static component of our error bounds is a polynomial in the machine epsilon $\varepsilon$, so an element of $F\left[ \varepsilon \right]$.
The runtime component of our semi-static error bounds is an expression in input values $x_1,\ldots,x_m$ and constants with the operators $\oplus, \ominus, \odot$ and $\left|\cdot\right|$.
We will call the set of such expression $F'\left[x_1,\ldots,x_m\right]$.
We will define two \emph{error bound maps} $E$ and $E_{\rm UFP}$ for all subexpressions $\tilde{q}$
of $\tilde{p}$ of the form 
\[
E,E_{\text{UFP}} : F\left[x_1,\ldots,x_m\right]\rightarrow F\left[\varepsilon\right]\times F'\left[x_1,\ldots,x_m\right] , \quad\tilde{q}\mapsto\left(a,m\right),
\]
such that the following invariants hold: 

Either 
\begin{equation}
m\left(x_{1},\ldots,x_{m}\right)\in\left\{ \infty,\text{NaN}\right\} \tag{I1} \label{eq:invariant2eq1}
\end{equation}
or both
\begin{equation}
\left|\tilde{q}\left(x_{1},\ldots,x_{m}\right)\right|\leq m\left(x_{1},\ldots,x_{m}\right), \tag{I2.1} \label{eq:invariant2eq2}
\end{equation}
and 
\begin{equation}
\left|\tilde{q}\left(x_{1},\ldots,x_{m}\right)-q\left(x_{1},\ldots,x_{m}\right)\right|\leq a\left(\varepsilon\right)\cdot m\left(x_{1},\ldots,x_{m}\right). \tag{I2.2} \label{eq:invariant2eq3}
\end{equation}

$E_{\text{UFP}}$, where UFP signifies underflow protection, will be constructed such that these invariants hold regardless of whether underflow occurs during any of the computations.
For $E$ this will not be guaranteed.
Because the polynomial $a\left(\cdot\right)$ will only
be evaluated in $\varepsilon$, we will omit the argument and will
use the polynomial and its value in $\varepsilon$ interchangeably.
The error bound maps are defined through a list of recursive \emph{error bound rules},
\[
R_{i\left(, \text{UFP} \right)}: F\left[x_1,\ldots,x_m\right]\rightarrow F\left[\varepsilon\right]\times F'\left[x_1,\ldots,x_m\right]
\]

for $1\leq i \leq 7$ as follows:

\begin{defn}\label{def:error-bound-rules} (Error Bound Rules, Error Bound Map)
Let $\tilde{q}:F^{m}\rightarrow F$ be a subexpression of a floating-point
polynomial $\tilde{p}:F^{m}\rightarrow F$. We define the following
error bound rules:

\begin{enumerate}
\item For a $\tilde{q}$ of the form $\tilde{q}\left(x_{1},\ldots,x_{m}\right)=c$
for some $c\in F$, we set 
\[
R_{1}\left(\tilde{q}\right) \coloneqq \left(0,\left|c\right|\right).
\]
\item For a $\tilde{q}$ of the form $\tilde{q}\left(x_{1},\ldots,x_{m}\right)=x_{i}$
for some $1\leq i\leq m$, we set 
\[
R_{2}\left(\tilde{q}\right) \coloneqq \left(0,\left|x_{i}\right|\right).
\]
\item For a $\tilde{q}$ of the form $\tilde{q}\left(x_{1},\ldots,x_{m}\right)=x_{i}\circledcirc x_{j}$
for some $1\leq i,j\leq m$ and $\circledcirc\in\left\{ \oplus,\ominus\right\} $,
we set 
\[
R_{3}\left(\tilde{q}\right) \coloneqq \left(\varepsilon,\left|x_{i}\circledcirc x_{j}\right|\right).
\]
\item For a $\tilde{q}$ of the form $\tilde{q}\left(x_{1},\ldots,x_{m}\right)=x_{i}\odot x_{j}$
for some $1\leq i,j\leq m$, we set 
\[
R_{4}\left(\tilde{q}\right) \coloneqq \left(\varepsilon,\left|x_{i}\odot x_{j}\right|\right).
\]
and
\[
R_{4,\text{UFP}}\left(\tilde{q}\right) \coloneqq \left(\varepsilon,\left|x_{i}\odot x_{j}\right|\oplus u_{N}\right).
\]
\item For a $\tilde{q}$ of the form $\tilde{q}\left(x_{1},\ldots,x_{m}\right)=\left(x_{i}\circledcirc_{1}x_{j}\right)\odot\left(x_{h}\circledcirc_{2}x_{g}\right)$
for some $1\leq g,h,i,j\leq m$ and $\circledcirc_{1},\circledcirc_{2}\in\left\{ \oplus,\ominus\right\} $,
we set 
\[
R_{5}\left(\tilde{q}\right) \coloneqq \left(3\varepsilon-\left(\phi-14\right)\varepsilon^{2},\left|\left(x_{i}\circledcirc_{1}x_{j}\right)\odot\left(x_{h}\circledcirc_{2}x_{g}\right)\right|\right)
\]
and
\[
R_{5,\text{UFP}}\left(\tilde{q}\right) \coloneqq \left(3\varepsilon-\left(\phi-14\right)\varepsilon^{2},\left|\left(x_{i}\circledcirc_{1}x_{j}\right)\odot\left(x_{h}\circledcirc_{2}x_{g}\right)\right|\oplus u_{N}\right)
\]
with 
\[
\phi\coloneqq2\left\lfloor \frac{-1+\sqrt{4\varepsilon^{-1}+45}}{4}\right\rfloor .
\]
\item For a $\tilde{q}$ of the form $\tilde{q}\left(x_{1},\ldots,x_{m}\right)=\tilde{q}_{1}\left(x_{1},\ldots,x_{m}\right)\circledcirc\tilde{q}_{2}\left(x_{1},\ldots,x_{m}\right)$
with $\circledcirc\in\left\{ \oplus,\ominus\right\} $, we set 
\[
R_{6}\left(\tilde{q}\right) \coloneqq \left(\left(1+\varepsilon\right)\max\left(a_{1},a_{2}\right)+\varepsilon,m_{1}\oplus m_{2}\right)
\]
and
\[
R_{6,\text{UFP}}\left(\tilde{q}\right) \coloneqq \left(\left(1+\varepsilon\right)\max\left(a_{1},a_{2}\right)+\varepsilon,m_{1}\oplus m_{2}\right)
\]
with $\left(a_{i},m_{i}\right)\coloneqq E\left(\tilde{q}_{i}\right)$ and $\left(a_{i},m_{i}\right)\coloneqq E_{\text{UFP}}\left(\tilde{q}_{i}\right)$ respectively 
for $i=1,2$. 
\item For a $\tilde{q}$ of the form $\tilde{q}\left(x_{1},\ldots,x_{m}\right)=\tilde{q}_{1}\left(x_{1},\ldots,x_{m}\right)\odot\tilde{q}_{2}\left(x_{1},\ldots,x_{m}\right)$,
we set 
\[
R_{7}\left(\tilde{q}\right) \coloneqq \left(\left(1+\varepsilon\right)\left(a_{1}+a_{2}+a_{1}a_{2}\right)+\varepsilon,m_{1}\odot m_{2}\right)
\]
and
\[
R_{7,\text{UFP}}\left(\tilde{q}\right) \coloneqq \left(\left(1+\varepsilon\right)\left(a_{1}+a_{2}+a_{1}a_{2}\right)+\varepsilon,m_{1}\odot m_{2}\oplus u_{N}\right)
\]
with $\left(a_{i},m_{i}\right)\coloneqq E\left(\tilde{q}_{i}\right)$ and  $\left(a_{i},m_{i}\right)\coloneqq E_{\text{UFP}}\left(\tilde{q}_{i}\right)$ respectively for $i=1,2$. 
\end{enumerate}

We define $E\left(\tilde{q}\right)$
to be the first applicable map out of $R_{1},R_{2},R_{3},R_{4},R_{5},R_{6}$
and $R_{7}$ and analogously $E_{\text{UFP}}\left(\tilde{q}\right)$ with the respective UFP-variations of the rules.
\end{defn}
It is straightforward to see that $E$ and $E_{\text{UFP}}$ are well-defined because the rules are exhaustive in the sense that there is no subexpression
in a floating-point polynomial for which no rule is applicable and
any recursion through $R_{6}$ and $R_{7}$ or their UFP-variations terminates at the level of individual variables.
\begin{rem}
For two polynomials $a_{1},a_{2}$ evaluated in $\varepsilon$ with
all coefficients smaller than $\varepsilon^{-1}$, $\max\left(a_{1},a_{2}\right)$
can be obtained by lexicographic comparison of coefficients of terms
in ascending order, i.e. starting with the linear term, since there
are no constant terms in the polynomials that occur.
\end{rem}
\begin{lem}
\label{lem:E-no-uf}Let $\tilde{p}$ be an arbitrary floating-point
polynomial then the invariants for $E\left(\tilde{q}\right)$
hold for every subexpression $\tilde{q}$ of $\tilde{p}$ for every
choice of floating-point inputs $x_{1},\ldots,x_{m}\in F$ such that
no underflow occurs in the evaluation of any subexpression of $\tilde{q}$.
\end{lem}

Following \cite{Shewchuk1997}, we introduce the following convenient
notation that will be used in the proof. We extend the arithmetic operations $\circ$ to sets $A,B\subset\R$ by $A\circ B:=\{a\circ b \mid a\in A, b\in B\}$, identify $a\in\R$ with $\{a\}$ for $\circ \in \{+,-,\cdot \}$, and set $A\pm a := A + [-a,a]$.

\begin{proof}
For any subexpression $\tilde{q}$ to which $R_{1}$ or $R_{2}$ applies,
the statement is obvious. For subexpressions for which $R_{3}$ or
$R_{4}$ are the first applicable rules and no overflow occurs,
(\ref{eq:invariant2eq2}) holds by the definition of $m$ and (\ref{eq:invariant2eq3})
follows from the definition of floating-point rounding. If overflow
occurs, $m$ is infinity and (\ref{eq:invariant2eq1}) holds. For subexpressions,
to which $R_{5}$ applies, either (\ref{eq:invariant2eq1}) holds
if overflow occurs or, if no overflow occurs, (\ref{eq:invariant2eq2})
holds by definition and (\ref{eq:invariant2eq3}) is proven in \cite{Ozaki2016}
in Lemma 3.1.

If $R_{6}$ is the first applicable rule, we assume that the invariant
(\ref{eq:invariant2eq1}) or the invariants (\ref{eq:invariant2eq2})
and (\ref{eq:invariant2eq3}) hold for $\left(a_{1},m_{1}\right)\coloneqq E\left(\tilde{q}_{1}\right)$
and $\left(a_{2},m_{2}\right)\coloneqq E\left(\tilde{q}_{2}\right)$
and we consider the case that $\tilde{q}=\tilde{q}_{1}\oplus\tilde{q}_{2}$.
If $\tilde{q}_{1}$ or $\tilde{q}_{2}$ is either $\infty$ or NaN,
by the assumption so are $m_{1}$ or $m_{2}$ and consequently $m_{1}\oplus m_{2}$
and (\ref{eq:invariant2eq1}) holds. If no overflow occurs, we see
that 
\begin{align*}
\left|\tilde{q}\right| & =\left|\tilde{q}_{1}\oplus\tilde{q}_{2}\right|\\
 & \leq\left|\tilde{q}_{1}\right|\oplus\left|\tilde{q}_{2}\right|\\
 & \leq m_{1}\oplus m_{2}
\end{align*}
and 
\begin{align*}
\tilde{q} & =\tilde{q}_{1}\oplus\tilde{q}_{2}\\
 & \in \tilde{q}_{1}+\tilde{q}_{2}\pm\varepsilon\left|\tilde{q}_{1}\oplus\tilde{q}_{2}\right|\\
 & \subseteq \tilde{q}_{1}+\tilde{q}_{2}\pm\varepsilon\left(m_{1}\oplus m_{2}\right)\\
 & \subseteq q_{1}\pm a_{1}\left(\varepsilon\right)m_{1}+q_{2}\pm a_{2}\left(\varepsilon\right)m_{2}\pm\varepsilon\left(m_{1}\oplus m_{2}\right)\\
 & \subseteq q\pm\max\left(a_{1}\left(\varepsilon\right),a_{2}\left(\varepsilon\right)\right)\left(m_{1}+m_{2}\right)\pm\varepsilon\left(m_{1}\oplus m_{2}\right)\\
 & \subseteq q\pm\left(\max\left(a_{1}\left(\varepsilon\right),a_{2}\left(\varepsilon\right)\right)\left(1+\varepsilon\right)+\varepsilon\right)\left(m_{1}\oplus m_{2}\right),
\end{align*}
where we used the assumption that the invariant holds for the two
subexpressions and standard floating-point rounding error estimates.
The proof for $\tilde{q}=\tilde{q}_{1}\ominus\tilde{q}_{2}$ is analogous.

If $R_{7}$ is the first applicable rule, we assume that the invariant
holds for $\left(a_{1},m_{1}\right)\coloneqq E\left(\tilde{q}_{1}\right)$
and $\left(a_{2},m_{2}\right)\coloneqq E\left(\tilde{q}_{2}\right)$.
Analogous to above, the case of overflow is trivial, so we consider
the case that no overflow occurs. Again it holds that 
\begin{align*}
\left|\tilde{q}\right| & =\left|\tilde{q}_{1}\odot\tilde{q}_{2}\right|\\
 & \leq m_{1}\odot m_{2}.
\end{align*}
and 
\begin{align*}
\tilde{q} & =\tilde{q}_{1}\odot\tilde{q}_{2}\\
 & \in \tilde{q}_{1}\cdot\tilde{q}_{2}\pm\varepsilon\left|\tilde{q}_{1}\odot\tilde{q}_{2}\right|\\
 & \subseteq \tilde{q}_{1}\cdot\tilde{q}_{2}\pm\varepsilon\left(m_{1}\odot m_{2}\right)\\
 & \subseteq \left(q_{1}\pm a_{1}m_{1}\right)\cdot\left(q_{2}\pm a_{2}m_{2}\right)\pm\varepsilon\left(m_{1}\odot m_{2}\right)\\
 & \subseteq q_{1}q_{2}\pm a_{2}m_{2}q_{1}\pm a_{1}m_{1}q_{2}\pm a_{1}a_{2}m_{1}m_{2}\pm\varepsilon\left(m_{1}\odot m_{2}\right)\\
 & \subseteq q_{1}q_{2}\pm a_{2}\left(1+\varepsilon\right)\left(m_{1}\odot m_{2}\right)\pm a_{1}\left(1+\varepsilon\right)\left(m_{1}\odot m_{2}\right)\\
 & \quad \pm a_{1}a_{2}\left(1+\varepsilon\right)\left(m_{1}\odot m_{2}\right)\pm\varepsilon\left(m_{1}\odot m_{2}\right)\\
 & \subseteq q_{1}q_{2}\pm\left(\left(1+\varepsilon\right)\left(a_{1}+a_{2}+a_{1}a_{2}\right)+\varepsilon\right)\left(m_{1}\odot m_{2}\right).
\end{align*}
Because the recursion eventually terminates at a non-recursion case
(rules 1--5), the claims for $E\left(\tilde{q}_{1}\right)$
and $E\left(\tilde{q}_{2}\right)$ hold. 
\end{proof}
\begin{lem}
\label{lem:E-uf}Let $\tilde{p}$ be an arbitrary floating-point polynomial
and then the invariants for $E_{\rm{UFP}}\left(\tilde{q}\right)$ hold for
every subexpression $\tilde{q}$ of $\tilde{p}$ for every choice
of inputs $x_{1},\ldots,x_{m}\in F$.
\end{lem}

\begin{proof}
Because it is useful for the parts of the proof that apply to the recursive rules $R_6$ and $R_7$, we will also prove for each rule applied to a subexpression $\tilde{q}$ that

\begin{equation}
q=\tilde{q} \quad \lor \quad  m \geq u_N \tag{I3} \label{eq:invariant3}
\end{equation}

holds.

For $R_{1}$ and $R_{2}$ there is nothing to prove.

For $R_{3}$ the reasoning given in the proof for Lemma~\ref{lem:E-no-uf} still applies
because the assumption of no underflow occurring was not used. If underflow occurs then $\tilde{q}$ is evaluated exactly, i.e. $\tilde{q} = q$ and if no underflow occurs then $\tilde{q}$ is not subnormal and hence it holds that $m \geq u_N$.

For $R_{4,\text{UFP}}$, we first note that $m$ is always non-zero and not smaller than either $\tilde{q}$ or $u_N$ so invariants~\eqref{eq:invariant2eq2} and \eqref{eq:invariant3} hold. Invariant~\eqref{eq:invariant2eq3} then follows directly from $2\varepsilon u_{N}=u_{S}$ and 
Lemma~\ref{def:floating-point-ops}.

\eqref{eq:invariant2eq3} for $R_{5,\text{UFP}}$ was proven as Lemma 3.1 in \cite{Ozaki2016}. For \eqref{eq:invariant2eq2} and \eqref{eq:invariant3} the same reasoning as for $R_{r,\text{UFP}}$ applies.

For the recursive rules $R_{6,\text{UFP}}$ and $R_{7,\text{UFP}}$, we assume at all invariants 
hold for the respective subexpressions $\tilde{q}_1$ and $\tilde{q}_2$, this is again justified because all recursions will be cases of rules $R_{1}$ to $R_{5,\text{UFP}}$ for which the invariants were already proven to hold or to other cases for rules $R_{6,\text{UFP}}$ and $R_{7,\text{UFP}}$.

For $R_{6,\text{UFP}}$, as in Lemma~\ref{lem:E-no-uf}, it is obvious that invariant~\eqref{eq:invariant2eq2} holds. If either $m_{1}$ or $m_{2}$ is equal
to or greater than $u_{N}$, then no underflow can occur in the evaluation
of $m$ and the invariant holds as proven in Lemma~\ref{lem:E-no-uf} and $m$ is greater than or equal to $u_{N}$.
If both $m_1$ and $m_2$ are smaller than $u_{N}$ then $\tilde{q}_{1}$ and $\tilde{q}_{2}$ are evaluated
error-free and $\tilde{q}$ is error-free if underflow occurs. If no underflow occurs in the evaluation of $\tilde{q}$,
the invariant also holds as proven in Lemma~\ref{lem:E-no-uf} and in this case $m$
is equal to or greater than $u_{N}$.

For $R_{7,\text{UFP}}$, we first note that, as in $R_{4,\text{UFP}}$, $m$ is greater than or equal to both $\tilde{q}$ and $u_N$, so invariants~\ref{eq:invariant2eq2} and \eqref{eq:invariant3} hold.
To show that \eqref{eq:invariant2eq3} holds, we use Lemma~\ref{def:floating-point-ops}
to obtain
\begin{align*}
\tilde{q} & =\tilde{q}_{1}\odot\tilde{q}_{2}\\
 & \in \tilde{q}_{1}\cdot\tilde{q}_{2}\pm\varepsilon\left|\tilde{q}_{1}\odot\tilde{q}_{2}\right|\pm\frac{1}{2}u_{S}\\
 & \subseteq \tilde{q}_{1}\cdot\tilde{q}_{2}\pm\varepsilon\left(m_{1}\odot m_{2}\right)\pm\frac{1}{2}u_{S}\\
 & \subseteq \tilde{q}_{1}\cdot\tilde{q}_{2}\pm\varepsilon\left(m_{1}\odot m_{2}+u_{N}\right)\\
 & \subseteq \left(q_{1}\pm a_{1}m_{1}\right)\cdot\left(q_{2}\pm a_{2}m_{2}\right)\pm\varepsilon\left(m_{1}\odot m_{2}+u_{N}\right)\\
 & \subseteq q_{1}q_{2}\pm a_{2}m_{2}q_{1}\pm a_{1}m_{1}q_{2}\pm a_{1}a_{2}m_{1}m_{2}\pm\varepsilon\left(m_{1}\odot m_{2}+u_{N}\right)\\
 & \subseteq q_{1}q_{2}\pm a_{2}\left(1+\varepsilon\right)\left(m_{1}\odot m_{2}\right)\pm a_{1}\left(1+\varepsilon\right)\left(m_{1}\odot m_{2}\right)\\
 & \pm a_{1}a_{2}\left(1+\varepsilon\right)\left(m_{1}\odot m_{2}\right)\pm\varepsilon\left(m_{1}\odot m_{2}+u_{N}\right)\\
 & \subseteq q_{1}q_{2}\pm\left(\left(1+\varepsilon\right)\left(a_{1}+a_{2}+a_{1}a_{2}\right)+\varepsilon\right)\left(m_{1}\odot m_{2}+u_{N}\right),
\end{align*}
which completes the proof.
\end{proof}

\subsection{Floating-Point Filters}

The following result provides two semi-static filters for floating-point
predicates that evaluate the sign of a polynomial. It is only stated
for floating-point realisations of polynomials that are sums or differences.
For products, the signs of each factor could be obtained individually
and then multiplied.
\begin{thm}
Let $p\in\mathbb{R}\left[x_{1},\ldots,x_{m}\right]$ be a polynomial
and $\tilde{p}\in F\left[x_{1},\ldots,x_{m}\right]$ be some floating-point
realisation of $p$ of the form $\tilde{p}=\tilde{p}_{1}\oplus\tilde{p}_{2}$
or $\tilde{p}=\tilde{p}_{1}\ominus\tilde{p}_{2}$.
\begin{enumerate}
\item Let $\left(a_{1},m_{1}\right)\coloneqq E\left(\tilde{p}_{1}\right)$
and $\left(a_{2},m_{2}\right)\coloneqq E\left(\tilde{p}_{2}\right)$.
Moreover, let constants $a_3,a_4 \in F$ satisfy
\[
    a_3 > \frac{\max\left(a_{1},a_{2}\right)}{1-\epsilon}, \qquad  a_4 \ge a_{3}\left(1+\varepsilon\right)^{2},
\]
and define 
\[
e\left(x_{1},\ldots,x_{m}\right)\coloneqq a_{4}\odot\left(m_{1}\left(x_{1},\ldots,x_{m}\right)\oplus m_{2}\left(x_{1},\ldots,x_{m}\right)\right).
\]
Then, for every choice of $x_{1},\ldots,x_{m}\in F\backslash\left\{ \mathrm{NaN},\infty,-\infty\right\} $
such that no underflow occurs in the evaluation of $\tilde{p}$ or
$e,$
\[
f\left(x_{1},\ldots,x_{m}\right)\coloneqq\begin{cases}
\mathrm{sign}\left(\tilde{p}\left(x_{1},\ldots,x_{m}\right)\right) & \left|\tilde{p}\right|>e\lor e=0\\
\mathrm{uncertain} & \text{otherwise}
\end{cases}
\]
is a valid filter.
\item Let $\left(a_{1},m_{1}\right)\coloneqq E_{{\rm UFP}}\left(\tilde{p}_{1}\right)$
and $\left(a_{2},m_{2}\right)\coloneqq E_{{\rm UFP}}\left(\tilde{p}_{2}\right)$.
We set $a_{3}$ and $a_{4}$ as in 1. and 
\[
e\left(x_{1},\ldots,x_{m}\right)\coloneqq a_{4}\odot\left(m_{1}\left(x_{1},\ldots,x_{m}\right)\oplus m_{2}\left(x_{1},\ldots,x_{m}\right)\right)\oplus u_{{\rm S}}.
\]
Then for every choice of $x_{1},\ldots,x_{m}\in F\backslash\left\{ \mathrm{NaN},\infty,-\infty\right\} $
\[
f\left(x_{1},\ldots,x_{m}\right)\coloneqq\begin{cases}
\mathrm{sign}\left(\tilde{p}\left(x_{1},\ldots,x_{m}\right)\right) & \left|\tilde{p}\right|>e\\
\mathrm{uncertain} & \text{otherwise}
\end{cases}
\]
is a valid filter.
\end{enumerate}
Note that $\left|\tilde{p}\right|>e$ always evaluates as false if $e$ is $\infty$ or
$\mathrm{NaN}$.
\end{thm}

\begin{proof}
We first prove 1. and we assume without loss of generality that $\tilde{p}=\tilde{p}_{1}\oplus\tilde{p}_{2}$.
Using Lemma~\ref{lem:fpn-error-bounds} and Lemma~\ref{lem:E-no-uf}, it holds that 
\begin{align*}
\tilde{p} & =\tilde{p}_{1}\oplus\tilde{p}_{2}\\
 & \subseteq\left(\tilde{p}_{1}+\tilde{p}_{2}\right)\pm\varepsilon\tilde{p}\\
 & \subseteq p\pm\varepsilon\tilde{p}\pm a_{1}m_{1}\pm a_{2}m_{2}
\end{align*}
and equivalently 
\begin{align*}
p & \in\tilde{p}\pm\varepsilon\tilde{p}\pm a_{1}m_{1}\pm a_{2}m_{2}\\
 & \subseteq\tilde{p}\pm\varepsilon\tilde{p}\pm\max\left(a_{1},a_{2}\right)\left(m_{1}+m_{2}\right).
\end{align*}
From this it follows that the signs of $p$ and $\tilde{p}$ are equal
if
\begin{equation}
\left(1-\varepsilon\right)\left|\tilde{p}\right|>\max\left(a_{1},a_{2}\right)\left(m_{1}+m_{2}\right).\label{eq:thm-proof-sufficient}
\end{equation}
The inequality
\begin{equation}
\left|\tilde{p}\right|>a_{3}\left(m_{1}+m_{2}\right)\label{eq:thm-proof-sufficient-2}
\end{equation}
is equivalent to 
\[
\left(1-\varepsilon\right)\left|\tilde{p}\right|>a_{3}\left(1-\varepsilon\right)\left(m_{1}+m_{2}\right),
\]
which implies (\ref{eq:thm-proof-sufficient}), so (\ref{eq:thm-proof-sufficient-2})
is also a sufficient condition. Lastly, we see that 
\begin{align*}
a_{3}\left(m_{1}+m_{2}\right)\leq & a_{3}\left(1+\varepsilon\right)\left(m_{1}\oplus m_{2}\right)\\
\leq & \left(a_{3}\left(1+\varepsilon\right)^{2}\right)\odot\left(m_{1}\oplus m_{2}\right)\\
\leq & a_{4}\odot\left(m_{1}\oplus m_{2}\right),
\end{align*}
where the second step uses the no-underflow-assumption. Hence, 
\[
\left|\tilde{p}\right|>a_{4}\odot\left(m_{1}\oplus m_{2}\right)\eqqcolon e
\]
is a sufficient condition for the signs of $p$ and $\tilde{p}$ being
equal. It remains to consider the case $e=0$. If $a_{1}$ and $a_{2}$
are both zero, then $\tilde{p}$ is a simple expression and its sign
is trivially correct, which makes $f$ always valid. If either $a_{1}$
or $a_{2}$ is non-zero, then $a_{4}$ is easily seen to be non-zero
too and $e$ can only be zero if both $m_{1}$ and $m_{2}$ are zero,
since we assumed that no underflow occurs. If $m_{1}=m_{2}=0$, then
by Lemma~\ref{lem:E-no-uf}, $\tilde{p}_{1}=\tilde{p}_{2}=p_{1}=p_{2}=0$, and $\tilde{p}=p=0$.

The proof for 2. is analogous except for the constant $u_{\mathrm{S}}$
being added to $e$ in place of assuming no underflow occurring,
the omittance of the case $e=0$, which can not occur in 2 and the usage of Lemma~\ref{lem:E-uf} in place of Lemma~\ref{lem:E-no-uf}.
\end{proof}
\begin{rem}
Note that the constants $a_{3}$ and $a_{4}$ do not depend on the
input but only on the expression of $\tilde{p}$ so in practice 
they can be computed at compile-time in floating point or exact arithmetic. 

\end{rem}

This filter differs from the semi-static filters (called stage A)
in \cite{Shewchuk1997}, which do not guarantee valid results in cases
of underflow. It also differs from the semi-static filters generated
by FPG \cite{meyer:inria-00344297} because only a single condition
is evaluated, rather than three conditions, which means that most
predicate calls for non-degenerate inputs can be decided on a code
path with a single, well-predictable branch.

With these two properties, having only a single branch on the filter success
code path and validity for inputs that can cause underflow, this
procedure to construct semi-static filters can be seen as a generalisation
of the semi-static 2D orientation filter presented in \cite{Ozaki2016}.
For the 2D orientation predicate, in particular, our approach produces
a more pessimistic error bound than \cite{Ozaki2016}, which could be considered as the price to pay for using a more general method.

The semi-static error bound $e$
can be turned into a static error bound by evaluating $m_{1}\oplus m_{2}$
not in specific input values $x_{1},\ldots,x_{m}$ but in bounds on
these values, $\left[\underline{x}_{1},\overline{x}_{1}\right],\ldots,\left[\underline{x}_{m},\overline{x}_{m}\right]$
using interval arithmetic, or by obtaining its maximum over some more
general domain in $F^{m}$. This yields a static or almost static
filter.

\subsection{Zero-Filter}
With underflow protection, the right-hand side of our semi-static filter condition
will never be zero, hence the filter will always fail, as in returning
``uncertain'', if the true sign of $p$ is $0$. For inputs in $F$
that approximate a uniform distribution on an interval in $\mathbb{R}$,
$p=0$ is extremely unlikely, but in some practical input data, it
might be more common. 

\begin{example}
Consider the 2D orientation predicate with the floating-point realisation
\[
\tilde{p}=\left(a_{x}\ominus c_{x}\right)\odot\left(b_{y}\ominus c_{y}\right)\ominus\left(a_{y}\ominus c_{y}\right)\odot\left(b_{x}\ominus c_{x}\right).
\]
It is easy to check that if either point $a$ or $b$ coincides with
point $c$ or if all points share the same $x$ or $y$ coordinate
and no overflow occurs, then $\tilde{p}$ evaluates to zero.
Such cases can be common degeneracies in real-world data.

It can also be checked that the error bound $e$ from our semi-static filter
without underflow-protection would be zero in either of these cases, so such degeneracies can be decided quickly by the filter.
The error bound of the $\mathrm{UFP}$-variation of our filter, though,
would not zero because non-zero terms would be introduced in the error
bounds of the multiplications and in the definition of $e$ itself. 
\end{example}

In this case, a simple filter that can certify
common cases for inputs that produce zeroes, can be useful. Such a
filter can be produced using the following rules.
\begin{defn}
(Zero-Filter) Let $\tilde{p}$ be a floating-point realisation of
a polynomial and let $x_{1},\ldots,x_{m}\in F$ a given set of input values.
We define the following rules.
\begin{enumerate}
\item For a subexpression $\tilde{q}$ of the form $\tilde{q}=c$ for some
constant $c\in F$ or input value $x_{i}$ for $i\in\left\{ 1,\ldots,m\right\} $,
we define 
\[
Z_{1}\left(\tilde{q};x_{1},\ldots,x_{m}\right)=\begin{cases}
\text{true}, & c=0\\
\text{false}, & \text{otherwise}.
\end{cases}
\]
\item For a subexpression $\tilde{q}$ of the form $\tilde{q}=x_{i}\circledcirc x_{j}$
for $1\leq i,j\leq m$ and $\circledcirc\in\left\{ \oplus,\ominus\right\} $,
we define 
\[
Z_{2}\left(\tilde{q};x_{1},\ldots,x_{m}\right)=\begin{cases}
\text{true}, & \tilde{q}=0\\
\text{false}, & \text{otherwise}.
\end{cases}
\]
\item For a subexpression $\tilde{q}$ of the form $\tilde{q}=\tilde{q}_{1}\circledcirc\tilde{q}$
for $1\leq i,j\leq m$ and $\circledcirc\in\left\{ \oplus,\ominus\right\} $,
we define 
\[
Z_{3}\left(\tilde{q};x_{1},\ldots,x_{m}\right)
= Z\left(\tilde{q}_{1};x_{1},\ldots,x_{m}\right)\land Z\left(\tilde{q}_{2};x_{1},\ldots,x_{m}\right).
\]
\item For a subexpression $\tilde{q}$ of the form $\tilde{q}=\tilde{q}_{1}\odot\tilde{q}_{2}$, we define
\[
Z_{4}\left(\tilde{q};x_{1},\ldots,x_{m}\right) 
=  Z\left(\tilde{q}_{1};x_{1},\ldots,x_{m}\right)\lor Z\left(\tilde{q}_{2};x_{1},\ldots,x_{m}\right).
\]
\end{enumerate}
We define $Z\left(\tilde{q};x_{1},\ldots,x_{m}\right)$ to be the
result of the first applicable rule out of $Z_{1},Z_{2},Z_{3},Z_{4}$.
\end{defn}

The zero-filter returns the sign $0$ if $Z\left(\tilde{p};x_{1},\ldots,x_{m}\right)$
is true and ``uncertain'' otherwise. It is easy to verify that
this filter is valid for all inputs regardless of range issues such
as overflow or underflow.

\section{Numerical Results}

The exact predicates derived in the previous section are designed to be fast, applicable to general polynomial expressions, and simple to use. These goals must be reflected in their implementation, which is briefly covered before benchmark results are presented.

\subsection{C++ Implementation} \label{sec:metaprogramming}

Our implementation of exact predicates is based on C++ template and
constexpr metaprogramming, making use of the Boost.Mp11 library~\cite{MP11}. The main design goal is the avoidance of
runtime overhead like the one that was seen in the C++-wrapper implementation in~\cite{burnikel_exact_2001} because geometric predicates can be found on
performance-critical code paths in geometric algorithms and can make up a large
proportion of overall runtime as the benchmarks in Section~\ref{subsec:benchmarks}  show. Further design goals include flexibility and
extensibility with regard to the choice and order of filters as well
as expressivity and simplicity in the definition of predicate expressions.

Exact predicates are implemented as variadic class templates for staged
predicates that hold a tuple of zero or more stages implementing filters
or the exact arithmetic evaluation. If all filters are semi-static, the instantiated
class is stateless and can be constructed without arguments and with no
runtime cost. The static parts of semi-static error bounds are computed
at compile-time from the predicate expression and static type information
for the calculation type. If almost static or static filters are included,
input bounds need to be provided at construction for the computation
of error bounds. For almost static filters, an update member function
is provided to update error bounds. The exact predicate is called
through a variadic function that takes a variable but compile-time
static number of inputs in the calculation type and returns an integer
out of $-1,0$ and $1$, that represents the result sign.

The individual stages are expected to follow the same basic interface.
Each stage provides at least a member function that is called with
input values and returns an integer that represents either the result
sign or a constant that indicates uncertainty. For stages that require
the computation of runtime constants, e.g. static and almost static
filters, constructor and update members need to be implemented as
well. Otherwise, the stages are default constructed at no runtime cost.
This general interface allows users of the library to extend exact
predicates with custom filters beyond those provided by our implementation
to better suit their algorithms and data sets, such as the filter shown in Example~\ref{exa:incircle-rect}.

\begin{lstlisting}
using ssf = semi_static_filter</* ... */>;
using es = predicate_approximation </* ... */, CGAL::Gmpzf>;
// This stage is exact because it uses an exact number type.

staged_predicate<ssf, es> pred;
// default constructed and stateless 2-stage predicate

int sign = pred.apply(ax, ay, bx, ...);
// exact value of the predicate p(a,b,..)
\end{lstlisting}

\begin{example} \label{exa:incircle-rect}
The 2D \predicate{incircle} predicate on four 2D points $p_{1},\ldots,p_{4}$
decides whether $p_{4}$ lies inside, on or outside of the oriented
circle passing through $p_{1},p_{2}$ and $p_{3}$, assuming the
points do not lie on a line. A pattern of degenerate inputs are four
points that form a rectangle. For this input, $p_{4}$ clearly lies
on the circle (indicated by a sign of $0$) but a forward error bound
filter could classify the case as undecidable and forward it to a
computationally expensive exact stage. The following listing illustrates
a custom filter that conforms to the previously described interface
and could be used with our implementation of staged predicates.
\begin{lstlisting}
struct incircle_rect_filter
{
  // stateless, no constructor or update method required
  template <typename CalculationType>
  static inline int apply(CalculationType ax, /* ... */)
  {
    if( (ax == bx && by == cy && cx == dx && dy == ay) ||
        /* ... */ )
      return 0;
    else
      return sign_uncertain;
  }
};
\end{lstlisting}
\end{example}

At the core of the implementation is the compile-time processing of
polynomial expressions for the derivation of error bound expressions.
Arithmetic expressions are represented in the C++ type system using
expression templates, a technique described in \cite{veldhuizen1995expression}. 
The most basic expressions in our implementation are types
representing the leaves of expression trees. Those leaves are either
compile-time constants (indexed with zero) or input values (indexed
with a positive number). More complex expressions can be built from
these placeholders using the elementary operators +, - and {*}.

Forward error bound expression types are deduced at compile-time based
on a list of rule class templates. The interface of each rule class
template requires a constexpr function that expects an expression
template and returns a bool indicating whether the rule is applicable
to the expression, and a class template for the error bound based on
the rule. Error bounds are implemented in the form of constexpr integer
arrays that represent the coefficients of the polynomial in $\varepsilon$
and a magnitude expression template. The rules can be extended through
custom rules that conform to the interface.

\begin{lstlisting}
constexpr auto orient2d = 
  (_1 - _5) * (_4 - _6) - (_3 - _5) * (_2 - _6);
// expression template representing the 2D orientation 
// predicate expression where _1, _2, _3, ... are 
// placeholders for ax, ay, bx, ...

using ssf = semi_static_filter<
  orient2d,
  forward_error_bound_expression<
    orient2d,
    double,
    /* ... rules ... */>
  >;
// a shorter alias for this construct is provided
\end{lstlisting}

\begin{example}
Consider a 2D orientation problem for points whose coordinates are
not binary floating-point numbers, e.g. because the input is given
in a decimal or rational format. The rules given in definition \ref{def:error-bound-rules} 
are not designed for this problem but with a custom
error bound rule, our implementation can be extended to generate a
filter for inputs that are rounded to the nearest floating-point number.
Such a filter could be used before going into a more computationally
expensive stage operating on decimal or rational numbers. 
\begin{lstlisting}
struct rounded_input
{
  template <typename Expression, /* ... */>
  static constexpr bool applicable()
  {
    if constexpr (Expression::is_leaf)
      return Expression::argn > 0;
    else
      return false;
  }

  template <typename Expression, /* ... */>
  struct error_bound
  {
    using magnitude = abs<Expression>;
    static constexpr std::array<long, 3> a
      {1, 0, 0};
	// the entries represent coefficients 
	// of the polynomial in eps
  };
};
\end{lstlisting}
The listing illustrates a custom rule. It is only applicable for expressions
that are input values, i.e.\ expressions of the form $\tilde{q}\left(x_{1},\ldots,x_{n}\right)=x_{i}$.
In the context of our implementation these expressions are leaves
of the expression tree with a positive index, and the error bound
is $R\left(\tilde{q}\right)=\left(\varepsilon,\left|x_{i}\right|\right)$. 
Using a rule set consisting of this custom rule, $R_{6,0}$ and $R_{7,0}$
on the 2D orientation predicate, yields the semi-static error bound
\[
\left(5\varepsilon\oplus32\varepsilon^{2}\right)\left(\left(\left|a_{x}\right|\oplus\left|c_{x}\right|\right)\odot\left(\left|b_{y}\right|\oplus\left|c_{y}\right|\right)\oplus\left(\left|a_{y}\right|\oplus\left|c_{y}\right|\right)\odot\left(\left|b_{x}\right|\oplus\left|c_{x}\right|\right)\right).
\]
\end{example}

Besides forward error bound based filters discussed in this paper, our implementation also contains templates for filters and exact stages based on the same principles as the stages B and D in~\cite{Shewchuk1997}.

\subsection{Benchmarks} \label{subsec:benchmarks}

To test the performance of our approach and implementation, we measured
timings for a number of benchmarks that are provided by the CGAL library.
The design of the 2D and 3D geometry kernels concept in CGAL as documented
in \cite{cgal:bfghhkps-lgk23-21a} provide  a simple way to test our predicates
in CGAL algorithms by deriving from the \lstinline{Simple_cartesian<double>}
kernel and overriding all predicate objects that may suffer from rounding
errors with predicates generated from our implementation.

The performance with the resulting custom kernel is then compared
to the performance of CGAL's \lstinline{Exact_predicates_inexact_constructions_kernel},
which follows a similar paradigm of filtered, exact predicates.

All benchmarks were run on a GNU/Linux workstation with a
Intel Core i7-6700HQ CPU using the performance scaling governor, no
optional mitigations against CPU vulnerabilities such as Spectre or
Meltdown, and disabled turbo for consistency. All code was compiled
with GCC 11.1 and the flags ``-O3 -march=native''.
The installed versions of relevant libraries were CGAL 5.4, GMP 6.2.1, MPFR 4.1.0, and Boost 1.79.

\subsubsection{2D Delaunay Triangulation}
The 2D Delaunay Triangulation algorithm provided by the CGAL library
makes use of the 2D \predicate{orientation} and \predicate{incircle} predicates, which compute
the sign of the following expressions: 
\[
\begin{aligned}p_{\text{orientation\_2}} & =\left|\begin{array}{cc}
a_{x}-c_{x} & a_{y}-c_{y}\\
b_{x}-c_{x} & b_{y}-c_{y}
\end{array}\right|\\
p_{\text{incircle\_2}} & =\left|\begin{array}{ccc}
a_{x}-d_{x} & a_{y}-d_{y} & \left(a_{y}-d_{y}\right)^{2}+\left(a_{y}-d_{y}\right)^{2}\\
b_{x}-d_{x} & b_{y}-d_{y} & \left(b_{y}-d_{y}\right)^{2}+\left(b_{y}-d_{y}\right)^{2}\\
c_{x}-d_{x} & c_{y}-d_{y} & \left(c_{y}-d_{y}\right)^{2}+\left(c_{y}-d_{y}\right)^{2}
\end{array}\right|.
\end{aligned}
\]
2D Delaunay Triangulations were computed for two data sets of randomly generated points.
The coordinates were sampled either from a continuous uniform distribution (using CGAL's Random\_\allowbreak points\_\allowbreak in\_\allowbreak square\_2 generator) or from an equidistant grid (using CGAL's points\_\allowbreak on\_\allowbreak square\_\allowbreak grid\_2 generator) and shuffled with 1,000,000 points in each data set. For the continuous distribution,
we found a 4.2\% performance penalty
for the use of underflow guards, which can be explained by the slightly
more expensive error expressions. With or without underflow guards,
our implementation performed faster than the CGAL predicates, see (a) in Fig.~\ref{fig:benchmark-plot}. This is expected because all calls can be decided on a code path with a
single, well-predictable branch.

\begin{figure}
\begin{center}
\includegraphics[width=0.7\columnwidth]{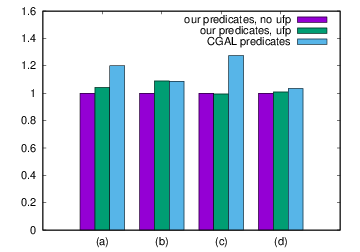}
\end{center}
\caption{This chart shows the relative runtime for the construction of a Delaunay triangulation of 1,000,000 points with coordinates sampled from (a) a continuous uniform distribution and (b) an equidistant grid, as well as for CGAL mesh (c) polygon processing and (d) refinement benchmarks. For each benchmark, our filters with and without underflow protection are compared to the predicates implemented in CGAL.
\label{fig:benchmark-plot}}
\end{figure}

For the points sampled from the equidistant grid, the triangulations were, in general, much slower, which
is also expected because the input is designed to be degenerate and trigger
edge cases. Our predicates with underflow
protection and the predicates in CGAL show very similar performance (roughly
0.2\% difference), while our filter without underflow protection is significantly
faster, see (b) in Fig.~\ref{fig:benchmark-plot}.

By construction, our semi-static filter with underflow protection
fails for all cases in which the true sign is zero, most of which
can be decided by the zero-filter, though. Table~\ref{tab:grid-delaunay} shows the number of filter failures
in the first filters for each predicate.
\begin{table}
\begin{center}
\begin{tabular}{ccccc}
\toprule
filter failures in the first stage & no UFP & UFP & CGAL & total calls\tabularnewline
\midrule
2D \predicate{orientation} & 49,375 & 624,400 & 49,641 & 4,121,216\tabularnewline
2D \predicate{incircle} & 1,112,461 & 1,490,010 & 1,641,255 & 8,455,667\tabularnewline
\bottomrule
\end{tabular}
\end{center}
\caption{Number of filter failures for the 2D \predicate{orientation} and 2D \predicate{incircle} predicate with various semi-static filters when constructing the Delaunay triangulation of 1,000,000 points sampled from an equidistant grid.}
\label{tab:grid-delaunay}
\end{table}
For a graphical comparison of the precision of 2D orientation filters, see Fig.~\ref{fig:precision}.

\begin{figure}[h]

\begin{subfigure}{0.5\textwidth}
\includegraphics[width=0.9\linewidth, height=3cm]{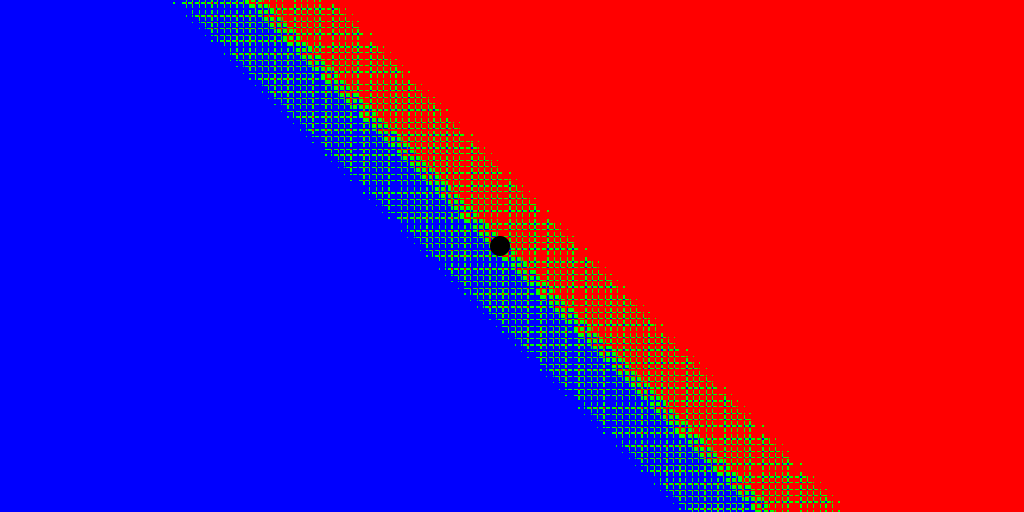} 
\caption{Naive predicate}
\end{subfigure}
\begin{subfigure}{0.5\textwidth}
\includegraphics[width=0.9\linewidth, height=3cm]{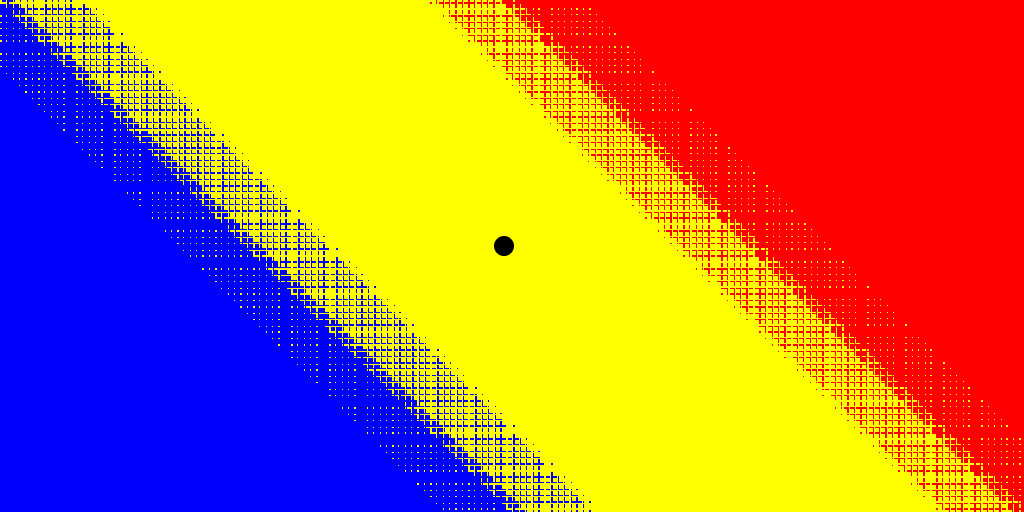}
\caption{Our filter}
\end{subfigure}
\begin{subfigure}{0.5\textwidth}
\includegraphics[width=0.9\linewidth, height=3cm]{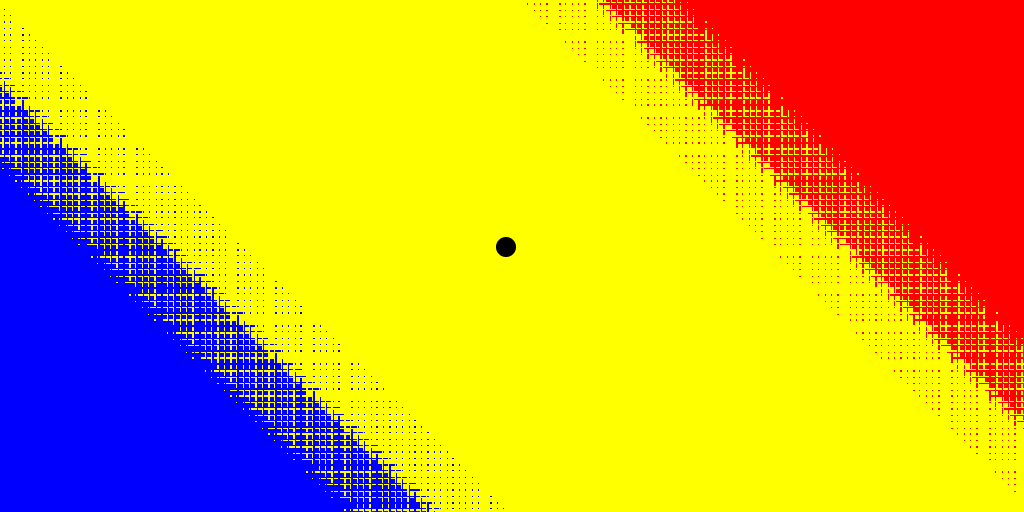} 
\caption{FPG filter}
\end{subfigure}
\begin{subfigure}{0.5\textwidth}
\includegraphics[width=0.9\linewidth, height=3cm]{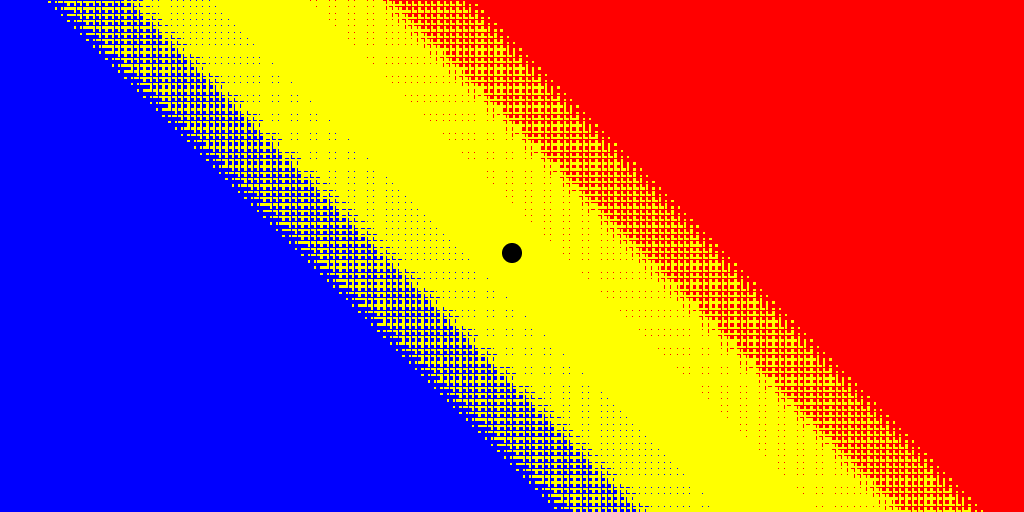}
\caption{Interval filter}
\end{subfigure}

\caption{This figure shows the result of calls to the non-robust 2D orientation predicate and to three 2D orientation filters respectively for the points (20.1, 20.1), (18.9, 18.9) and a small neighbourhood of the point (3.5, 3.5), such that neighbouring pixels represent points with neighbouring floating-point coordinates. The point (3.5, 3.5) is marked with a black circle. The dimensions of the neighbourhood are roughly $6\cdot10^{-13}$ in width and $3\cdot10^{-13}$ in height. The colours represent left side (red), collinear (green), right side (blue) and uncertain (yellow). The pattern of green points in (a) shows that the naive predicate produces many incorrect results. Our filter (b) is more precise than FPG (c) but less precise than the significantly slower interval filter. Ozaki's filter produces the exact same image as our filter.}
\label{fig:precision}
\end{figure}

\subsubsection{3D Polygon Mesh Processing}

The next benchmark was taken from the Polygon Mesh processing benchmark
in CGAL. For this benchmark, first, a 3D mesh is taken, and a polyhedral
envelope with a distance $\delta$ is taken around it. The polyhedral
envelope is an approximation of the Minkowski sum of the mesh with
a sphere, also known as a buffer. Then, three points are repeatedly
chosen in a loop, and if they form a non-degenerate triangle, it is
tested whether that triangle is contained in the polyhedral envelope
or not. As input, we use the file pig.off, which is provided as a
sample in the CGAL tree, and for $\delta$ we chose 0.1. This is described
in more detail in~\cite{cgal:lty-pmp-22a}.

The algorithm makes use of the 3D \predicate{orientation} predicate defined
as the sign of 
\[
\begin{aligned}p_{\text{orientation\_3}} & =\left|\begin{array}{ccc}
a_{x}-d_{x} & a_{y}-d_{y} & a_{z}-d_{z}\\
b_{x}-d_{x} & b_{y}-d_{y} & b_{z}-d_{z}\\
c_{x}-d_{x} & c_{y}-d_{y} & c_{z}-d_{z}
\end{array}\right|\end{aligned}
.
\]
 No filter failures were recorded for either implementation, and no
performance penalty was measured for the underflow protection. The predicates
provided by CGAL caused an additional runtime of around 28\% compared
to our implementation, see (c) in Fig.~\ref{fig:benchmark-plot}.

\subsubsection{3D Mesh Refinement}

As the last benchmark, we measure the runtime of 3D mesh refinement
with CGAL. The algorithm and its parameters are explained in~\cite{cgal:rty-m3-22a}. 
The predicates used in this benchmark are the
3D \predicate{orientation} predicate and the \predicate{power side of oriented power sphere}
predicate, which is defined as the sign of the following expression
\[
\begin{aligned}p & =\left|\begin{array}{cccc}
a_{x}-e_{x} & a_{y}-e_{y} & a_{z}-e_{z} & \left(a_{x}-e_{x}\right)^{2}+\left(a_{y}-e_{y}\right)^{2}+\left(a_{z}-e_{z}\right)^{2}+\left(e_{w}-a_{w}\right)\\
b_{x}-e_{x} & b_{y}-e_{y} & b_{z}-e_{z} & \left(b_{x}-e_{x}\right)^{2}+\left(b_{y}-e_{y}\right)^{2}+\left(b_{z}-e_{z}\right)^{2}+\left(e_{w}-b_{w}\right)\\
c_{x}-e_{x} & c_{y}-e_{y} & c_{z}-e_{z} & \left(c_{x}-e_{x}\right)^{2}+\left(c_{y}-e_{y}\right)^{2}+\left(c_{z}-e_{z}\right)^{2}+\left(e_{w}-c_{w}\right)\\
d_{x}-e_{x} & d_{y}-e_{y} & d_{z}-e_{z} & \left(d_{x}-e_{z}\right)^{2}+\left(d_{x}-e_{z}\right)^{2}+\left(d_{x}-e_{z}\right)^{2}+\left(e_{w}-d_{w}\right)
\end{array}\right|\end{aligned}
,
\]
which has with $d=5$ the highest degree of all predicates used in
our benchmarks and is based on a non-homogeneous polynomial. As input file,
we used elephant.off, which is provided as a sample in the CGAL source
tree, with a face approximation error of 0.0068, a max facet sign
of 0.003 and a maximum tetrahedron size of 0.006.

The underflow guard came with a slight performance penalty of around
1\%, and the CGAL predicates were about 3.4\% slower, see (d) in Fig.~\ref{fig:benchmark-plot}. There was a non-zero but negligible number of filter failures of around 0.1\% for each
of the predicates.

\section*{Conclusion}
We have presented a recursive scheme for the derivation of (semi-)static filters for geometric predicates. The approach is branch-efficient, sufficiently general to handle rounding errors, overflow and underflow and can be applied to arbitrary polynomials.

Our C++-metaprogramming-based implementation is user-friendly in so far as it requires no code generation tools, additional annotations for variables or manual tuning. This is achieved without the additional runtime overhead of previous C++-wrapper-based implementations, and our measurements show that our approach is competitive with and even outperforms the state-of-the-art in some cases.

Future work could include generalisations toward non-polynomial predicates and robust predicates on implicit points that occur as results or interim results of geometric constructions and may not be explicitly representable with floating-point coordinates. The implementation may also be extended in the future to include further filtering stages to improve the performance for common cases of degenerate inputs.

\bibliographystyle{plain}
\bibliography{references}

\begin{thebibliography}{10}

\bibitem{cgal:rty-m3-22a}
Pierre Alliez, Cl{\'e}ment Jamin, Laurent Rineau, St{\'e}phane Tayeb, Jane
  Tournois, and Mariette Yvinec.
\newblock {3D} mesh generation.
\newblock In {\em {CGAL} User and Reference Manual}. {CGAL Editorial Board},
  {5.4} edition, 2022.

\bibitem{Attene}
Marco Attene.
\newblock Indirect predicates for geometric constructions.
\newblock {\em Computer-Aided Design}, 126:102856, 04 2020.

\bibitem{bartels21}
T.~Bartels and V.~Fisikopoulos.
\newblock Fast robust arithmetics for geometric algorithms and applications to
  gis.
\newblock {\em The International Archives of the Photogrammetry, Remote Sensing
  and Spatial Information Sciences}, XLVI-4/W2-2021:1--8, 2021.

\bibitem{Berg08}
Mark~de Berg, Otfried Cheong, Marc~van Kreveld, and Mark Overmars.
\newblock {\em Computational Geometry: Algorithms and Applications}.
\newblock Springer-Verlag TELOS, Santa Clara, CA, USA, 3rd ed. edition, 2008.

\bibitem{Bronnimann}
Herv\'{e} Br\"{o}nnimann, Christoph Burnikel, and Sylvain Pion.
\newblock Interval arithmetic yields efficient dynamic filters for
  computational geometry.
\newblock In {\em Proc.\ of the 14th Annual Symposium on Computational
  Geometry}, pages 165--174, USA, 1998. ACM.

\bibitem{cgal:bfghhkps-lgk23-21a}
Herv{\'e} Br{\"o}nnimann, Andreas Fabri, Geert-Jan Giezeman, Susan Hert,
  Michael Hoffmann, Lutz Kettner, Sylvain Pion, and Stefan Schirra.
\newblock {2D} and {3D} linear geometry kernel.
\newblock In {\em {CGAL} User and Reference Manual}. {CGAL Editorial Board},
  {5.2.1} edition, 2021.
\newblock \url{https://doc.cgal.org/5.2.1/Manual/packages.html\#PkgKernel23}.

\bibitem{burnikel_exact_2001}
Christoph Burnikel, Stefan Funke, and Michael Seel.
\newblock Exact geometric computation using cascading.
\newblock {\em International Journal of Computational Geometry \&
  Applications}, 11(03):245--266, June 2001.
\newblock Publisher: World Scientific Publishing Co.

\bibitem{Menezes2021FastPE}
Marcelo de~Matos~Menezes, Salles Viana~Gomes Magalhães, Matheus~Aguilar
  de~Oliveira, W.~Randolph Franklin, and Rodrigo~Eduardo de~Oliveira
  Bauer~Chichorro.
\newblock Fast parallel evaluation of exact geometric predicates on {GPU}s.
\newblock submitted, January 2021.

\bibitem{Devillers00}
Olivier Devillers, Alexandra Fronville, Bernard Mourrain, and Monique Teillaud.
\newblock Algebraic methods and arithmetic filtering for exact predicates on
  circle arcs.
\newblock In {\em Proc. of the 16th Annual Symposium on Computational
  Geometry}, page 139–147, USA, 2000. ACM.

\bibitem{Devillers03}
Olivier Devillers and Sylvain Pion.
\newblock Efficient exact geometric predicates for {Delaunay} triangulations.
\newblock In Richard~E. Ladner, editor, {\em Proceedings of the Fifth Workshop
  on Algorithm Engineering and Experiments, Baltimore, MD, USA, January 11,
  2003}, pages 37--44. {SIAM}, 2003.

\bibitem{MP11}
Peter Dimov.
\newblock Boost {C++} libraries: Mp11, version 1.76, 2021.
\newblock \url{https://boost.org/libs/mp11}.

\bibitem{fisikopoulos16}
Vissarion Fisikopoulos and Luis~Pe\ naranda.
\newblock Faster geometric algorithms via dynamic determinant computation.
\newblock {\em Computational Geometry}, 54:1--16, 2016.

\bibitem{BG}
Barend Gehrels, Bruno Lalande, Mateusz Loskot, Adam Wulkiewicz, Menelaos
  Karavelas, and Vissarion Fisikopoulos.
\newblock Boost {C++} libraries: Geometry, version 1.76, 2021.
\newblock \url{https://boost.org/libs/geometry}.

\bibitem{Granlund12}
Torbjörn Granlund and {the GMP development team}.
\newblock {\em {GNU MP}: {T}he {GNU} {M}ultiple {P}recision {A}rithmetic
  {L}ibrary}, 5.0.5 edition, 2012.
\newblock \url{http://gmplib.org/}.

\bibitem{hauser_handling_1996}
John~R. Hauser.
\newblock Handling floating-point exceptions in numeric programs.
\newblock {\em ACM Transactions on Programming Languages and Systems},
  18(2):139--174, March 1996.

\bibitem{cgal:hhkps-nt-22a}
Michael Hemmer, Susan Hert, Sylvain Pion, and Stefan Schirra.
\newblock Number types.
\newblock In {\em {CGAL} User and Reference Manual}. {CGAL Editorial Board},
  {5.4} edition, 2022.

\bibitem{Higham2002}
Nicholas~J. Higham.
\newblock {\em Accuracy and Stability of Numerical Algorithms}.
\newblock Society for Industrial and Applied Mathematics, jan 2002.

\bibitem{IEEE754}
754-2008 -- {IEEE} standard for floating-point arithmetic.
\newblock IEEE, 2008.

\bibitem{Jamin15}
Cl\'{e}ment Jamin, Pierre Alliez, Mariette Yvinec, and Jean-Daniel Boissonnat.
\newblock Cgalmesh: A generic framework for {Delaunay} mesh generation.
\newblock {\em ACM Trans. Math. Softw.}, 41(4), October 2015.

\bibitem{Kettner2004}
Lutz Kettner, Kurt Mehlhorn, Sylvain Pion, Stefan Schirra, and Chee Yap.
\newblock Classroom examples of robustness problems in geometric computations.
\newblock In {\em Algorithms {\textendash} {ESA} 2004}, pages 702--713.
  Springer Berlin Heidelberg, 2004.
\newblock \url{https://doi.org/10.1007/978-3-540-30140-0_62}.

\bibitem{Li2004}
Zhilin Li, Christopher Zhu, and Chris Gold.
\newblock {\em Digital Terrain Modeling: Principles and Methodology}.
\newblock {CRC} Press, nov 2005.
\newblock \url{https://doi.org/10.1201/9780203357132}.

\bibitem{cgal:lty-pmp-22a}
S{\'e}bastien Loriot, Mael Rouxel-Labb{\'e}, Jane Tournois, and Ilker~O. Yaz.
\newblock Polygon mesh processing.
\newblock In {\em {CGAL} User and Reference Manual}. {CGAL Editorial Board},
  {5.4} edition, 2022.

\bibitem{BM}
John Maddock and Christopher Kormanyos.
\newblock Boost {C++} libraries: Geometry, version 1.76, 2021.
\newblock \url{https://boost.org/libs/multiprecision}.

\bibitem{meyer:inria-00344297}
Andreas Meyer and Sylvain Pion.
\newblock {FPG: A code generator for fast and certified geometric predicates}.
\newblock In {\em {Real Numbers and Computers}}, pages 47--60, Santiago de
  Compostela, Spain, June 2008.
\newblock \url{https://hal.inria.fr/inria-00344297}.

\bibitem{Nanevski2003}
Aleksandar Nanevski, Guy Blelloch, and Robert Harper.
\newblock Automatic generation of staged geometric predicates.
\newblock {\em Higher-Order and Symbolic Computation (formerly {LISP} and
  Symbolic Computation)}, 16(4):379--400, December 2003.

\bibitem{Ozaki2016}
Katsuhisa Ozaki, Florian Bünger, Takeshi Ogita, Shin’ichi Oishi, and
  Siegfried~M. Rump.
\newblock Simple floating-point filters for the two-dimensional orientation
  problem.
\newblock {\em BIT Numerical Mathematics}, 56(2):729--749, 2016.

\bibitem{Meng19}
Meng Qi, Ke~Yan, and Yuanjie Zheng.
\newblock Gpredicates: Gpu implementation of robust and adaptive floating-point
  predicates for computational geometry.
\newblock {\em IEEE Access}, 7:60868--60876, 2019.

\bibitem{Rump2011}
Siegfried~M. Rump.
\newblock Error estimation of floating-point summation and dot product.
\newblock {\em {BIT} Numerical Mathematics}, 52(1):201--220, jun 2011.

\bibitem{shewchukimpl}
Jonathan Shewchuk.
\newblock Routines for arbitrary precision floating-point arithmetic and fast
  robust geometric predicates.
\newblock 1996.

\bibitem{Shewchuk1997}
Jonathan~Richard Shewchuk.
\newblock Adaptive precision floating-point arithmetic and fast robust
  geometric predicates.
\newblock {\em Discrete {\&} Computational Geometry}, 18(3):305--363, oct 1997.

\bibitem{Shewchuk98}
Jonathan~Richard Shewchuk.
\newblock Tetrahedral mesh generation by {Delaunay} refinement.
\newblock In {\em Proceedings of the Fourteenth Annual Symposium on
  Computational Geometry}, SCG '98, page 86–95, New York, NY, USA, 1998.
  Association for Computing Machinery.

\bibitem{Sunday21}
Daniel Sunday.
\newblock {\em Practical Geometry Algorithms: with {C++} Code}.
\newblock Amazon Digital Services LLC, 2021.
\newblock ISBN: 9798749449730.

\bibitem{veldhuizen1995expression}
Todd Veldhuizen.
\newblock Expression templates.
\newblock {\em C++ Report}, 7(5):26--31, 1995.

\end{thebibliography}

\end{document}